\documentclass[11pt,english,oneside]{amsart}
\usepackage[T1]{fontenc}
\usepackage[latin9]{inputenc}
\usepackage{geometry}
\usepackage{amssymb}
\usepackage{color}

\usepackage{graphicx,color}
\usepackage{amsmath, amssymb, graphics}

\usepackage{graphicx}

\makeatletter
\numberwithin{equation}{section} 
\numberwithin{figure}{section} 
  \@ifundefined{theoremstyle}{\usepackage{amsthm}}{}
  \theoremstyle{plain}
  \newtheorem{thm}{Theorem}[section]
    \theoremstyle{plain}
  \newtheorem{my-def}[thm]{Definition}
  \theoremstyle{plain}
  \newtheorem{cor}[thm]{Corollary}
  \theoremstyle{plain}
  \newtheorem{prop}[thm]{Proposition}
  \theoremstyle{remark}
  \newtheorem{rem}[thm]{Remark}
  \theoremstyle{remark}
  
  \theoremstyle{plain}
  
\usepackage{geometry}

\def\bfR#1{{\bf R}^#1}

\def\com#1{ \hbox{#1}}

\def\e{\hbox{\rm e}}

\smallskip
\def\<{{\langle }}
\def\>{{\rangle }}

\makeatother

\usepackage{babel}

\def\bfR#1{{\bf R}^#1}

\def\com#1{ \quad\hbox{#1}\quad}

\def\e{\hbox{\rm e}}

\smallskip
\def\<{{\langle }}
\def\>{{\rangle }}

\makeatother

\begin{document}

\title{The TreadmillSled of a curve}

\author{ Oscar M. Perdomo }

\date{\today}

\curraddr{Department of Mathematics\\
Central Connecticut State University\\
New Britain, CT 06050\\
}

\email{ perdomoosm@ccsu.edu}

\begin{abstract}
In \cite{P}, the author introduced the notion of TreadmillSled of a curve, which is an operator that takes regular curves in $\bfR{2}$ to curves in $\bfR{2}$. This operator turned out to be very useful to describe helicoidal surfaces, for example, it provides an interpretation for the profile curve of helicoidal surfaces with constant mean curvature similar to the well known interpretation of the profile curve of Delaunay's surfaces using conics, see \cite{P}. In \cite{KP} the authors used the TreadmillSled to classify all helicoidal surfaces with constant anisotropic mean curvature coming from axially symmetric anisotropic energy density. Also, in \cite{P}, the author proves that an helicoidal surface different from a cylinder has constant Gauss curvature if and only if  the TreadmillSled of its profile curve lies in a vertical semi line contained in the lower or upper half plane and not contained in the $y$-axis... Why not the whole vertical line? and why the semi-line cannot be contained in the $y$-axis?  In this paper we provide several properties of the TreadmillSled operator, in particular we will answer the questions in the previous sentence. Finally, we prove that the TreadmillSled of the profile curve of a minimal helicoidal surface is either a hyperbola or a the $x$-axis. The latter case occurs only when the surface is a helicoid.

\end{abstract}


\maketitle
\section{Introduction}

The surface of revolution generated by the regular curve $\alpha=(y(s),z(s))$ with $z(s)\ne0$  is given by

$$\phi(s,t)=(z(s)\sin(t),y(s),z(s)\cos(t))$$

The curve $\alpha$ is called the profile curve. In \cite{De}, Delaunay proved that a surface of revolution has constant mean curvature if and only if it is a sphere, a cylinder or if its profile curve lies in the trace made by the focus of a conic, when this conic rolls along the $y$-axis. When the conic used is a parabola, the surface is minimal and it is called catenoid; if the conic used is a hyperbola, the surface is called a nodoid and if the conic used is an ellipse, the surface is called an unduloid. Since an ellipse has two foci, the trace of each one of them generates an undoloid. It is not difficult to see that these two unduloids are essentially the same, one is a translation of the other. Figure 1.1 
shows how the profile curve of an unduloid is constructed using an ellipe.

\begin{figure}[h]\label{unduloidandprofilecurve}
\centerline{\includegraphics[width=5.59cm,height=6cm]{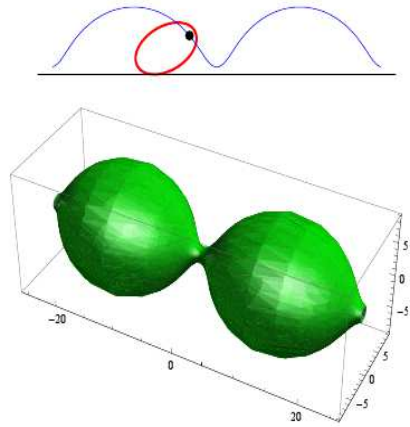}}
\caption{An unduloid and the construction of its profile curve }
\end{figure}

When we roll the parabola its focus traces a curve of infinity length. Figure 1.2 
shows a catenoid and its profile curve.

\begin{figure}[h]\label{catenoid}
\centerline{\includegraphics[width=7.7cm,height=5cm]{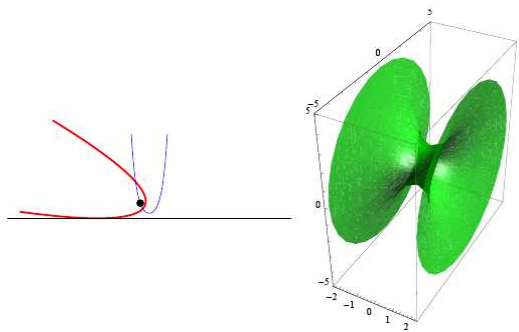}}
\caption{A catenoid and the construction of its profile curve }
\end{figure}

When we roll a set of hyperbolas, their foci trace two curves of finite length, each curve generates a cmc surface. It can be proven that we can translate one of these surfaces to obtain a smooth connected surface with constant mean curvature. If we repeat this connected piece over and over we obtain a  complete cmc surface. In Figure 1.3 
we show the trace of the foci, the two-piece cmc surface and the connected piece made by gluing the translation of one of the connected components of the initial two-piece surface to the other connected component.

\begin{figure}[h]\label{nodoid}
\centerline{\includegraphics[width=12.8cm,height=5cm]{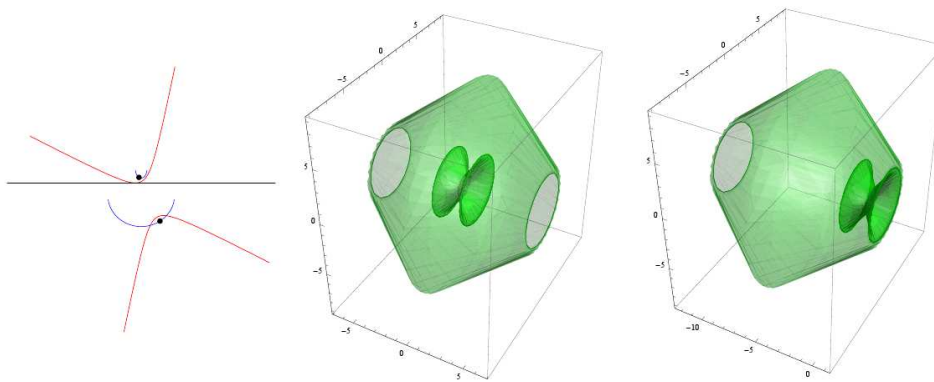}}
\caption{A nodoid and the construction of its profile curve}
\end{figure}

Let us restate Delaunay's theorem. Think of the operator $Roll$ that takes regular curves in $\bfR{2}$ into curves in $\bfR{2}$ given in the following way: For a regular curve $\alpha:[a,b]\to \bfR{2}$, let $s(t)$ denote the length of the curve from $\alpha(a)$ to $\alpha(t)$ and let us define

$$Roll(\alpha)= \{ T_t\begin{pmatrix}0 \\ 0\end{pmatrix}\, :\, \hbox{$T_t$ is an oriented isometry in $\bfR{2}$, $T_t(\alpha(t))=\begin{pmatrix}s(t) \\ 0\end{pmatrix}$ and $\, dT_t(\frac{\alpha^\prime(t)}{|\alpha^\prime(t)| })= \begin{pmatrix} 1 \\ 0 \end{pmatrix}$}\}$$

With this operator, Delaunay's theorem implies that if $\alpha:[0,l]\to \bfR{2}$ is a piece of conic with focus at the origin, then $Roll(\alpha)$ is the profile curve of a surface of revolution with constant mean curvature. Notice how the center of the curve $\alpha$ plays an important role in the definition of $Roll(\alpha)$.

In \cite{P}, the author found a dynamical interpretation for helicoidal surfaces with constant mean curvature. For the sake of comparison, let us rewrite the first part of the introduction, but this time for helicoidal surfaces.
The helicoidal surface generated by the regular curve $\alpha=(x(s),z(s))$  is given by

$$\phi(s,t)=(x(s) \cos(w t)+z(s) \sin(wt),t,-x(s)\sin(w t)+z(s)\cos(w t))\com{with $w>0$ fixed }$$

The curve $\alpha$ is called the profile curve. In \cite{P}, the author proved that a helicoidal surface has constant mean curvature one, if and only if, it is  a cylinder of radius $\frac{1}{2}$ or if the origin of its profile curve, $\alpha$, trace the curve

\begin{eqnarray}\label{heart-shaped curves}
 x^2 + y^2- \frac{y}{\sqrt{1+w^2x^2}}=M\com{for some $M> -\frac{1}{4}$}
\end{eqnarray}

when $\alpha$ moves on a treadmill located at the origin aligned in the direction of the $x$-axis. Figures 1.4 and 1.5 show examples of how the center of the profile curve generates curves of the form (\ref{heart-shaped curves}).

\begin{figure}[h]\label{catenoid}
\centerline{\includegraphics[width=9.41cm,height=5cm]{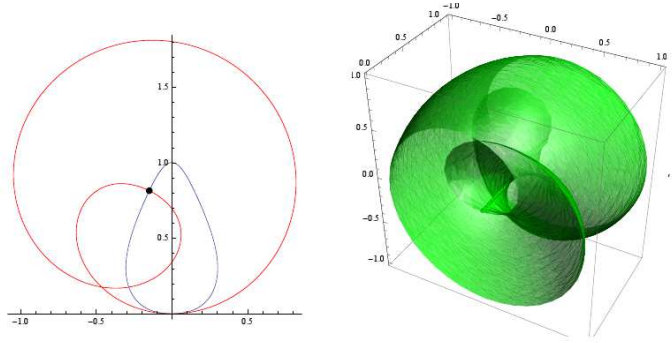}}
\caption{A helicoidal surface with cmc and the construction of its profile curve. The closed embedded curve on the left is given by the equation (\ref{heart-shaped curves})}
\end{figure}

\begin{figure}[h]\label{fundamental piece}
\centerline{\includegraphics[width=11.43cm,height=5cm]{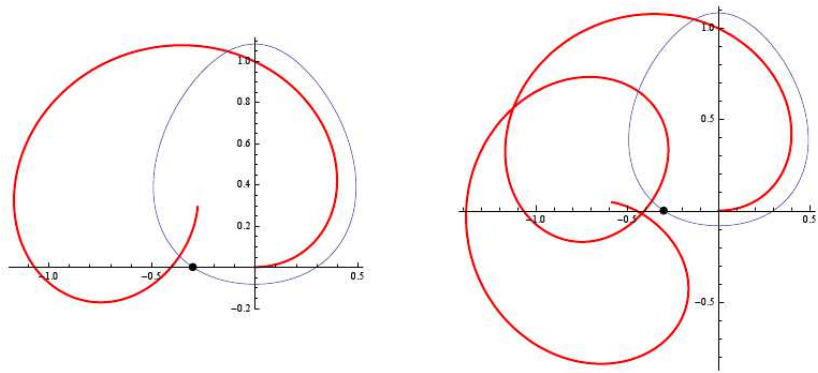}}
\caption{The profile curve of a cmc helicoidal surface is the union of fundamental pieces, here we show a fundamental piece on the left and the union of two fundamental pieces on the right.}
\end{figure}

The fact that the TreadmillSled of the profile curve of a helicoidal surface with cmc is a closed curve allow us to define a fundamental piece of the profile curve, see Figure 1.5, 
which in turn, easily provides a dense versus properly-immersed duality property for all these cmc surfaces. See \cite{P} for details.

Let us restate the previous theorem for helicoidal cmc surfaces. Think of the operator $TSS$ that take regular curves in $\bfR{2}$ into curves in $\bfR{2}$ given in the following way: for a regular curve $\alpha:[a,b]\to \bfR{2}$, let us define

$$TSS(\alpha)= \{ T_t\begin{pmatrix}0 \\ 0\end{pmatrix}\, :\, \hbox{$T_t$ is an oriented isometry in $\bfR{2}$, $T_t(\alpha(t))=\begin{pmatrix} 0 \\ 0\end{pmatrix}$ and $\, dT_t(\frac{\alpha^\prime(t)}{|\alpha^\prime(t)| })= \begin{pmatrix} 1 \\ 0 \end{pmatrix}$}\}$$

The letters $TSS$ stand for TreadmillSled Set. With this operator, we can say that if $\alpha:[a,b]\to \bfR{2}$ is the profile curve of a helicoidal surface with cmc one, then $TSS(\alpha)$ lies in a curve of the form (\ref{heart-shaped curves}). Notice how the center of curve $\alpha$ plays an important role in the definition of $TSS(\alpha)$. Figure 1.6 shows the TreadmillSled of the graph of a polynomial of degree 3.

\begin{figure}[h]\label{catenoid}
\centerline{\includegraphics[width=11.25cm,height=6cm]{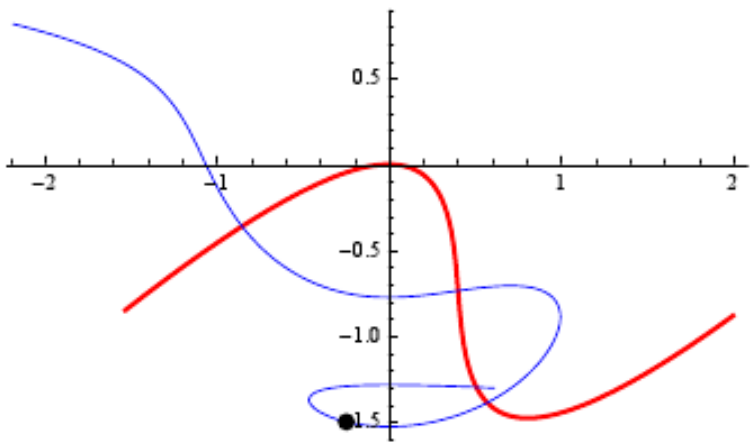}}
\caption{The TreadmillSled of the graph of a cubic polynomial }
\end{figure}

As more applications of the TreadmillSled, in  \cite{P}, the author showed that a helicoidal surface has zero Gauss curvature if and only if the TreadmillSled of its profile curve lies in a vertical semi line contained in the upper or lower plane (see Figure 1.7).

\begin{figure}[h]\label{flathelicoidal}
\centerline{\includegraphics[width=7.27cm,height=5cm]{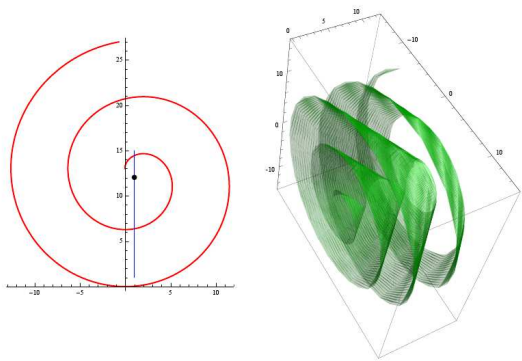}}
\caption{The TreadmillSled of the profile curve of a flat helicoidal surface lies in a vertical semiline. }
\end{figure}

Anytime we fixed a positive function $\upsilon:S^2\to {\bf R}$ on the Euclidean unit sphere we can define an {\it anisotropic} mean curvature for surfaces in $\bfR{3}$. When $\upsilon$ is the constant function $1$, the anisotropic mean curvature agrees with the mean curvature. Delaunay's result for constant mean curvature of revolution was extended to anisotropic constant mean curvature surfaces of revolution by Miyuki and Palmer in 2008, \cite{MP}. Likewise, the dynamical interpretation  for helicoidal constant mean curvature, \cite{P}, was extended to anisotropic constant mean curvature surfaces by  Khuns and Palmer \cite{KP}. In the latter paper, Khuns and Palmer found a formula for the inverse of the operator TreadmillSled. Here in this paper we will elaborate more on this formula and prove some properties for the TreadmillSled. For purposes of a better understanding we will find an explicit parametrization for the TreadmillSled of a curve and we will be refereing to this parametrization of the TreadmillSled as just $TS$; in this way, $TS$ becomes an operator that takes a parametrized regular curve into a parametrized curve. We will show that this operator acts like the derivative operator for functions. For example,

\begin{itemize}
\item

Given $\alpha:[a,b]\to \bfR{2}$, $TS(\alpha)$ is an expression of $\alpha$ and $\alpha^\prime$

\item

If $TS(\alpha)=TS(\beta)$ then $\alpha$ and $\beta$ differ by a constant. This time the constant does not represent a translation on the graph like in the case of the derivative operator but it represents an oriented rotation that fixes the origin. More precisely, if we identify $\bfR{2}$ with the complex numbers, then $\beta=\e^{ic}\alpha$ for some constant $c$.
\item

When $\gamma$ is in the image of the operator $TS$, there is a formula for $TS^{-1}(\gamma)$ that depends on $\gamma$, $\gamma^\prime$ and only one antiderivative. The ambiguity of this antiderivative is responsible for the existence of the whole 1 parametric family of curves with the same TreadmillSled.

\item
If we change the orientation of $\alpha$, that is, if we consider the curve $\beta(t)=\alpha(-t)$, then $TS(\beta)(t)=-TS(\alpha)(-t)$.

\end{itemize}

If we look at figures 1.4, 1.5 and 1.6, we notice that the curves that are TreadmillSled have the property that their velocity vector is horizontal where the curve intercepts the  $y$ axis. This is not a coincidence, actually we will show that a curve $\gamma(t)=(x(t),y(t))$ is the TreadmillSled of a regular curve $\alpha$ if and only if
\begin{itemize}
\item
$y^\prime(t)=-f(t) x(t)$ for some continuous function $f$ and
\item
$y(t) f(t)-x^\prime(t)$ is a positive function.
\end{itemize}

From the first property we see that if $x(t_0)=0$, then $y^\prime(t_0)=0$ and therefore the velocity vector $\gamma^\prime(t_0)$ is horizontal on points along the $y$-axis.  The second property is the reason why a whole vertical line cannot be the TreadmillSled of a regular curve. For a vertical line, $x^\prime(t)$ always vanishes and therefore when the line touches the $x$-axis, the function $y(t) f(t)-x^\prime(t)$ vanishes making the second property fail at this point. Also notice that if the vertical semi-line is contained in the $y$-axis, then by the relation  $y^\prime(t)=-f(t) x(t)$, we must have that $y^\prime(t)$ vanishes and therefore $\gamma$ reduces to just a point. It is easy to see that the TreadmillSled of a circle centered at the origin is just a point in the $y$-axis. Notice that if the profile curve of a helicoidal surface is a circle centered at the origin then the surface is a cylinder.

At the end of this paper we prove that the TreadmillSled of the profile curve of a minimal helicoidal surface is either the $x$-axis (when the surface is a helicoid), or it is a hyperbola centered at the origin.

\section{The $\phi$-TreadmillSled of a curve}

In this paper we will assume that all functions have as many derivatives as needed. Let us start this section with a definition that extends the notion of TreadmillSled. One of the reasons we introduce this notion is because it provides an interpretation for the curve $h(t)\alpha(t)$ when $h:[a,b]\to \mathbb{C}$, $\alpha:[a,b]\to \mathbb{C}$ are curve in the complex plane with $|h(t)|=1$, see Corollary \ref{complexmultbycurvesinS1}.

\begin{my-def}\label{phi Treadmill} Given a regular curve $\alpha:[a,b]\to \bfR{2}$ and a function $\phi:[a,b]\to {\bf R}$, we define the $\phi$ -TreadmillSled of $\alpha$ as the set of points

$$ \{ T_s\begin{pmatrix}0 \\ 0\end{pmatrix}\, :\, \hbox{$T_s$ is an oriented isometry in $\bfR{2}$, $T_s(\alpha(s))=\begin{pmatrix}0 \\ 0\end{pmatrix}$ and $\, dT_s(\frac{\alpha^\prime(s)}{|\alpha^\prime(s)| })= \begin{pmatrix}\cos(\phi(s)) \\ \sin(\phi(s))\end{pmatrix}$}\}$$

This set of points will be denoted by $ \phi\hbox{-}TS(\alpha)$.
\end{my-def}

\begin{rem} \label{independece of parameter} Notice that the definition of the $\phi\hbox{-}TS(\alpha)$ is independent of the parametrization, it only depends on the orientation of the curve. That is, if $h:[c,d]\to [a,b]$ is a function with positive derivative and $\tilde{\alpha}(t)=\alpha(h(t))$ and $\tilde{\phi}(t)=\phi(h(t))$, then $\tilde{\phi}\hbox{-}TS(\tilde{\alpha})=\phi\hbox{-}TS(\alpha)$
\end{rem}

It is not difficult to see that the $\phi$-TreadmillSled of $\alpha$ can be viewed as the curve generated by doing the following steps:

\begin{itemize}
\item
Imagine that the curve $\alpha$ is in a plane which can freely move. Moreover, let us assume that there is a hole in the origin of this plane and also let us assume that we have placed a pencil in this hole.

\item

Imagine that another plane, this one fixed, contains a treadmill based at the origin with a device that allows the treadmill to incline at any angle.

\item

The curve $\alpha$ in the moving plane will generate another curve in the fixed plane, the $\phi$-TreadmillSled of $\alpha$.

\item

The $\phi$-TreadmillSled of $\alpha$ is the curve drawn on the fixed plane by the pencil located at the origin of the moving plane, when the curve $\alpha$ passes on the treadmill with the property that, anytime the point $\alpha(s)$ is on the treadmill, the treadmill is aligned in the direction $(\cos(\phi(s)),\sin(\phi(s)))$.

\end{itemize}

The following proposition will provide a formula to find the $\phi$-TreadmillSled of a curve $\alpha$

\begin{prop}\label{formula for TS}
Let $\alpha:[a,b]\to \bfR{2}$ be a regular curve in $\bfR{2}$, if $\alpha(s)=(x(s),y(s))^T=\begin{pmatrix}x(s)\\y(s)\end{pmatrix}$, then,

\begin{eqnarray}\label{parametrization of TS}
\beta(s)=A(\theta(s))\,  \alpha(s)=-A(-\phi(s))A(\rho(s))\alpha(s)
\end{eqnarray}

is a parametrization of the $\phi$-TreadmillSled of $\alpha$. Here

$$ A(\tau)=\begin{pmatrix} \cos(\tau)& \sin(\tau) \\ -\sin(\tau)& \cos(\tau) \end{pmatrix}$$

and

$$\theta(s)=\rho(s)-\phi(s)+\pi\com{and} \begin{pmatrix}\cos(\rho(s))\\\sin(\rho(s)) \end{pmatrix}=\frac{1}{|\alpha^\prime(s)|}\, \alpha^\prime(s)$$

\end{prop}

\begin{proof}
We will use the parameter $s$ to describe points in the set $\phi\hbox{-}TS(\alpha)$. For a fixed $s\in [a,b]$, let us find an oriented isometry of $\bfR{2}$ such that $T_s(\alpha(s))=\begin{pmatrix}0 \\ 0\end{pmatrix}$ and $\, dT_s(\frac{\alpha^\prime(s)}{|\alpha^\prime(s)| })= \begin{pmatrix}\cos(\phi(s)) \\ \sin(\phi(s))\end{pmatrix}$. We know that

$$T_s\begin{pmatrix}u \\ v\end{pmatrix} = A(\tilde{\theta}(s))\begin{pmatrix}u \\ v\end{pmatrix} + \begin{pmatrix}c_1(s) \\ c_2(s)\end{pmatrix}$$

Notice that once we find $\tilde{\theta}(s)$, $c_1(s)$ and $c_2(s)$, using the definition \ref{phi Treadmill}, we get that $\beta(s)=(c_1(s),c_2(s))^T$ is a point in $\phi\hbox{-}TS(\alpha)$; and therefore, when we vary $s$ in the interval $[a,b]$, we obtain that  $\beta(s)=(c_1(s),c_2(s))$ is a parametrization of $\phi\hbox{-}TS(\alpha)$.

Since

$$dT_s \begin{pmatrix}v_1 \\ v_2\end{pmatrix}=A(\tilde{\theta}(s))\begin{pmatrix}v_1 \\ v_2\end{pmatrix}\com{ and}   dT_s(\frac{\alpha^\prime(s)}{|\alpha^\prime(s)| })= \begin{pmatrix}\cos(\phi(s)) \\ \sin(\phi(s))\end{pmatrix}$$

We have that

$$A(\tilde{\theta}(s))\begin{pmatrix}\cos(\rho(s))\\\sin(\rho(s)) \end{pmatrix}=\begin{pmatrix}\cos(\phi(s))\\\sin(\phi(s)) \end{pmatrix}$$

and therefore,

$$A(\tilde{\theta}(s))A(-\rho(s))\begin{pmatrix}1\\0 \end{pmatrix}=A(-\phi(s)) \begin{pmatrix}1\\ 0 \end{pmatrix}$$

Since $A(\tau_1+\tau_2)=A(\tau_1)A(\tau_2)$, the last equation implies that $A(\tilde{\theta}(s)-\rho(s)+\phi(s))\, \begin{pmatrix}1\\0 \end{pmatrix}=\begin{pmatrix}1\\0 \end{pmatrix}$, which implies that $\tilde{\theta}(s)=\rho(s)-\phi(s)$.

Now, using the equation $T_s(\alpha(s))=\begin{pmatrix}0\\0 \end{pmatrix}$ we get that

$$\begin{pmatrix}c_1(s)\\c_2(s) \end{pmatrix} = -A(\tilde{\theta}(s))\alpha(s)=A(\theta(s)) \alpha(s)$$

Since $\beta(s)=(c_1(s),c_2(s))^T$, then the proposition follows.

\end{proof}

\begin{rem}
The definition of TreadmillSled of a curve given in the introduction corresponds with the $\phi$-TreadmillSled when $\phi$ is the zero function. Sometimes we will view $\phi\hbox{-}TS(\alpha)$ not as a set but as the parametrized curve described in (\ref{parametrization of TS}).
\end{rem}

To be more precise, we will use Proposition \ref{formula for TS} to define the TreamillSled as an operator that takes a regular parametric curve into a parametric curve.

\begin{my-def} \label{definition of TS}
Let $\alpha:[a,b]\to \bfR{2}=\begin{pmatrix} x(s)\\ y(s)\end{pmatrix}$ be a regular curve. We define the TreadmillSled of $\alpha$ as the parametric curve $TS(\alpha):[a,b]\to \bfR{2}$ given by

$$TS(\alpha)(s)=\frac{1}{\sqrt{x^\prime(s)^2+y^\prime(s)^2}} \, \begin{pmatrix} -x^\prime(s) x(s)-y^\prime(s) y(s) \\ x(s)y^\prime(s)-y(s)x^\prime(s)\end{pmatrix} $$
\end{my-def}

\begin{cor}
If  we identify each point $\begin{pmatrix} x_1\\ x_2\end{pmatrix}\in \bfR{2}$ with the complex number $x_1+i\, x_2$, then

$$\phi\hbox{-}TS(\alpha)=\e^{\phi}\, TS(\alpha) $$

For any curve $\alpha(s)=x_1(s)+i\, x_2(s)$. Moreover, if the function $\phi$ is fixed, then, the $\phi$-TreadmillSled of two curves is the same, if and only if the TreadmillSled of the curves is the same.
\end{cor}

The following proposition gives us some insight about the nature of the operator $\phi$-TreadmillSled defined in the set of regular curves. As we already notice, the $\phi$-TreadmillSled of a curve is independent of the parametrization as long as the orientation is preserved. Therefore, there is not loss of generality if we assume that the curves in the domain of the operator $\phi$-TreadmillSled are parametrized by arc-length.

\begin{prop}
Let $\alpha_1:[a,b]\to \bfR{2}$ and $\alpha_2:[a,b]\to \bfR{2}$  be two curves parametrized by arc-length. $TS(\alpha_1)=TS(\alpha_2)$ if and only if $\alpha_2(s)=A(\tau)\alpha_1(s)$ for some constant $\tau$.
\end{prop}

\begin{proof}
Let $\rho_1(s)$ and $\rho_2(s)$ be functions such that $\begin{pmatrix}\cos(\rho_i(s))\\ \sin(\rho_i(s))\end{pmatrix}=\alpha_i^\prime(s)$. If $\alpha_2(s)=A(\tau)\alpha_1(s)$, then

$$\alpha_2^\prime(s)=A(\tau)\begin{pmatrix} \cos(\rho_1(s))\\ \sin(\rho_1(s)) \end{pmatrix}=\begin{pmatrix}\cos(\rho_1(s)-\tau)\\ \sin(\rho_1(s)-\tau)\end{pmatrix}  $$
therefore, we may assume that $\rho_2(s)=\rho_1(s)-\tau$. Using Proposition \ref{formula for TS}, we obtain that

$$TS(\alpha_2)=A(\rho_2+\pi)\alpha_2=A(\rho_1-\tau+\pi)A(\tau)\alpha_1=A(\rho_1+\pi)\alpha_1=TS(\alpha_1)$$

Therefore, we have proven that if $\alpha_2=A(\tau)\alpha_1$, then $TS(\alpha_1)=TS(\alpha_2)$. Now let us assume that
$TS(\alpha_1)=TS(\alpha_2)$. Let us fix an $s_0\in [a,b]$ such that $|\alpha_1(s_0)|\ne 0$.  Since $TS(\alpha_1)(s_0)=TS(\alpha_2)(s_0)$ then $|\alpha_1(s_0)|=|\alpha_2(s_0)|$. Let $\tau$ be a real number such that $A(\tau)\alpha_1(a)=\alpha_2(a)$ and let us consider $\alpha_3(s)=A(\tau)\alpha_1(s)$. That is, $TS(\alpha_3)=TS(\alpha_2)$, and moreover, we have that $\alpha_3(s_0)=\alpha_2(s_0)$. Using Definition \ref{definition of TS} we get that if $\alpha(s)=\begin{pmatrix} x(s)\\ y(s) \end{pmatrix} $ is a curve parametrized by arc-length and
$TS(\alpha)(s)=\begin{pmatrix} z(s)\\ w(s) \end{pmatrix} $, then,

\begin{eqnarray*}
z(s)&=& -x^\prime(s) x(s)-y^\prime(s) y(s)\cr
w(s)&=& x(s)y^\prime(s)-y(s)x^\prime(s)
\end{eqnarray*}

For values of $s$ such that $x(s)^2+y(s)^2>0$ we get that

\begin{eqnarray*}
x^\prime(s)&=& - \frac{1}{x(s)^2+y(s)^2} \, (  x(s) z(s)+y(s) w(s))\cr
y^\prime(s)&=& \frac{1}{x(s)^2+y(s)^2} \, (  x(s) w(s)-y(s) z(s))
\end{eqnarray*}

By the existence and uniqueness theorem of ordinary differential equations we get that the conditions $\alpha_3(s_0)=\alpha_2(s_0)$ and $TS(\alpha_3)=TS(\alpha_2)$ imply that $\alpha_2(s)=\alpha_3(s)$ for all $s$ near $s_0$. Since both curves are regular, by a continuity argument we conclude that the real number $\tau$ is independent of $s_0$ and therefore $\alpha_2(s)=\alpha_3(s)$ for all $s$. We then get $\alpha_2=A(\tau)\alpha_1$ for some $\tau$. This finishes the proof of the proposition.

\end{proof}

\begin{rem}\label{remark TS using J}
Let us define $J=\begin{pmatrix} 0& -1\\1&0\end{pmatrix}$. If $\alpha$ is an arc-length parametrized curve and $\begin{pmatrix} z(s) \\ w(s) \end{pmatrix}=TS(\alpha)$, then,

$$z=-\<\alpha,\alpha^\prime\>\com{and} w=\<\alpha^\prime,J\alpha\> \com{where $\<\, ,\,\>$ is the Euclidean inner product.}  $$

With this definition of $J$ we have that the curvature of $\alpha$ is $k(s)=\<\alpha^{\prime\prime}(s),J(\alpha^\prime(s))\>$. Notice that if $\alpha^\prime(s)=\begin{pmatrix}\cos(\rho(s)) \\ \sin(\rho(s))\end{pmatrix}$, then $k(s)=\rho^\prime(s)$. Therefore, if we know the curvature $k(s)$  of a curve parametrized by arc-Length and a given angle for the velocity vector, let's say $\alpha^\prime(a)=\begin{pmatrix}\cos(\rho_0) \\ \sin(\rho_0)\end{pmatrix}$, then

$$TS(\alpha)=-A(\rho(s))\alpha(s)\com{where} \rho(s)=\int_a^sk(u)du+\rho_0$$

\end{rem}

The following corollary provides a way to program the inclination on a treadmill (find the function $\phi$) if we want to get the curve $\e^{ig(t)}$ as the $\phi\,$-TreadmillSled of $\alpha$.

\begin{cor} \label{complexmultbycurvesinS1}If $\alpha:[a,b]\longrightarrow \mathbb{C}\cong\bfR{2}$ is a regular curve with curvature function $\kappa$ and $g:[a,b]\to {\bf R}$ is a function, then

$$\e^{i g(t)}\alpha(t)=\phi\hbox{-}TS(\alpha)\, $$

where $\phi(t)=\int_a^t\kappa(\tau)\, |\alpha^\prime(\tau)|\, d\tau+\rho_0 +g(t)+\pi$ and   $\alpha^\prime(a)=\begin{pmatrix}\cos(\rho_0) \\ \sin(\rho_0)\end{pmatrix}$

\end{cor}

We will now characterize the range of the operator TreadmillSled and find an inverse of this operator. Under the assumption that a curve $\gamma$ is in the range of the operator $TS$, the formula for the inverse of the TreadmillSled provided below was found in \cite{KP}.

\begin{prop} \label{inverse formula for TS}

Let us denote $\gamma(s)=\begin{pmatrix}z(s)\\ w(s) \end{pmatrix}$. $\gamma$ is the Treadmillsled of a regular curve $\alpha$ if and only if $w^\prime(s)=-f(s)z(s)$ for some continuous function $f$ and $wf-z^\prime>0$. More precisely, if $f,w$ and $z$ satisfy the two previous conditions, and $F(s)$ is an antiderivative of $f(s)$,  then,

$$TS(\alpha)=\gamma\com{where} \alpha(t)=-A(-F(t)) \gamma(t)$$

\end{prop}

\begin{proof} Let us assume that $\gamma(s)$ is the TreadmillSled of a curve $\alpha$. Let us first consider the case when $\alpha$ is parametrized by arc-Length. If we denote by $k_\alpha$ the curvature of $\alpha$, then, using Remark \ref{remark TS using J} we obtain,

$$z=-\<\alpha,\alpha^\prime\>\com{and} w=\<\alpha^\prime,J\alpha\>  $$

Therefore,

$$w^\prime=\<\alpha^{\prime\prime},J\alpha\> +\<\alpha^\prime,J\alpha^\prime\>=k\<J\alpha^{\prime},J\alpha\>=k_\alpha\<\alpha^{\prime},\alpha\>=-k_\alpha z$$

and

$$ z^\prime=-1-\<\alpha,\alpha^{\prime\prime}\>=-1-k_\alpha\<\alpha,J\alpha^\prime\>=-1+k_\alpha\<J\alpha,\alpha^\prime\>=-1+k_\alpha w$$

Taking $f=k_\alpha$ we conclude that $w^\prime=-f z$ and $f w-z^\prime = 1$. If we now consider a regular curve $\tilde{\alpha}$, then we have that $\tilde{\alpha}(t)=\alpha(h(t))$ where $\alpha$ is parametrized by arc-length and $h(t)$ is a function with $h^\prime(t)>0$. Therefore, by either Remark \ref{independece of parameter} or by Definition \ref{definition of TS}, we get that if $\tilde{\gamma}=\begin{pmatrix}\tilde{z}\\ \tilde{w} \end{pmatrix}$ is the TreadmillSled of $\tilde{\alpha}$, then
$\tilde{\gamma}(t)=\gamma(h(t))$ where $\gamma=\begin{pmatrix}z\\ w \end{pmatrix}$ is the TreadmillSled of $\alpha$. Since $\alpha$ is parametrized by arc-length, then $f=-\frac{w^\prime}{z}$ is continuous and $fw-z^\prime=1$. Therefore,
$\tilde{f}(t)=-\frac{\tilde{w}^\prime(t)}{\tilde{z}(t)}=h^\prime(t) f(h(t))$ is continuous and

$$\tilde{w}(t) \tilde{f}(t)-\tilde{z}^\prime(t)=h^\prime(t) w(h(t)) f(h(t)-h^\prime(t) z^\prime(h(t))=h^\prime(t)>0$$

This inequality finishes the proof of one of the implications of the Proposition. Let us assume now that the functions $z,w$ are given and that
$f=-\frac{w^\prime}{z}$ is continuous and that $fw-z^\prime=-\frac{w w^\prime}{z}-z^\prime >0$. We need to prove that if $F^\prime=f$ then

$$\alpha(t)=-A(-F(t))\gamma(t) $$

satisfies that $TS(\alpha)=\gamma$. Using the fact that

$$\frac{dA(\tau)}{d\tau}=A(\tau+\frac{\pi}{2})=A(\tau)A(\frac{\pi}{2})=-A(\tau)J$$

we get that

$$\alpha^\prime=-fA(-F)J\gamma-A(-F)\gamma^\prime=-A(-F)\, (fJ\gamma+\gamma^\prime )$$

Therefore,

$$\<\alpha^\prime,\alpha^\prime\>=\<fJ\gamma+\gamma^\prime ,fJ\gamma+\gamma^\prime \>=f^2\<\gamma,\gamma\>+2 f \<J\gamma,\gamma^\prime\>+
\<\gamma^\prime,\gamma^\prime\>$$

Since $f=-\frac{w\prime}{z}$, we get,

$$\<\alpha^\prime,\alpha^\prime\>=\frac{(w^\prime)^2}{z^2}(z^2+w^2)-2 \frac{w^\prime}{z}(z w^\prime-w z^\prime )+(z^\prime)^2+(w^\prime)^2=(z^\prime+\frac{ww^\prime}{z})^2 = (fw-z^\prime)^2$$

Since we have that $fw-z^\prime>0$, then we conclude that $|\alpha^\prime|=fw-z^\prime$ and therefore $\alpha$ is a regular curve.

\begin{eqnarray*}
TS(\alpha)&=& \frac{1}{|\alpha^\prime|} \begin{pmatrix} -\<\alpha^\prime,\alpha \> \\ \,  \< \alpha^\prime,J\alpha\>\end{pmatrix}
                    = \frac{1}{|\alpha^\prime|} \begin{pmatrix} -\<\gamma,fJ\gamma+\gamma^\prime\>\\ \<fJ\gamma+\gamma^\prime,J\gamma\>\end{pmatrix}
                    = \frac{1}{|\alpha^\prime|} \begin{pmatrix} -\<\gamma,\gamma^\prime\>\\ f \<\gamma,\gamma \> + \<\gamma^\prime,J\gamma\>\end{pmatrix}
\end{eqnarray*}

Since,

$$-\<\gamma,\gamma^\prime\>=-ww^\prime-zz^\prime=wfz-zz^\prime=z(wf-z^\prime)=z|\alpha^\prime|$$

and,

$$ f \<\gamma,\gamma \> + \<\gamma^\prime,J\gamma\>=-\frac{w^\prime}{z}(z^2+w^2)+zw^\prime-wz^\prime=  -\frac{w^2w^\prime}{z}-wz^\prime=w(fw-z^\prime)=w|\alpha^\prime|$$

we conclude that $TS(\alpha)=\gamma$. This completes the proof of the proposition.

\end{proof}
\section{A dynamical interpretation for helicoidal minimal surfaces}

Helicoidal minimal hypersurfaces have been understood for a long time. For a detailed study we refer to the last section of the last chapter of the book of Differential Geometry by Graustein \cite{G}. We have that all the isometry surfaces (except for the catenoid) from the well known family of surfaces that starts with a helicoid and ends with a catenoid are helicoidal minimal surfaces. Actually, every helicoidal minimal surface belongs to one of these families.

\begin{figure}[h]\label{helicoidal minimal}
\centerline{\includegraphics[width=8.93cm,height=5cm]{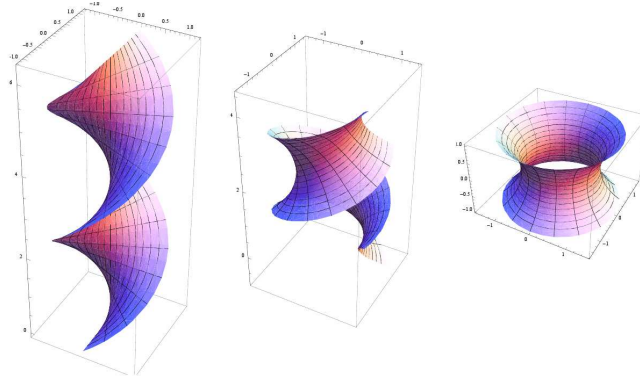}}
\caption{A helicoid (left), a helicoidal minimal surface (center) and a catenoid (right) are part of a family of isometry minimal surfaces.}
\end{figure}

A similar result for helicoidal cmc surfaces was proven in \cite{DD} by Do Carmo and Dajczer. They proved that every helicoidal surface belongs to a family of isometry surfaces that continuously move from an unduloid to a nodoid. In this section we provide a dynamical interpretation for the profile curve of a helicoidal minimal surface. Let us state and prove the main theorem in this section.

\begin{thm}\label{thm ts of helicoidal minimal}
A complete helicoidal surface $\phi(s,t)=(x(s) \cos(w t)+z(s) \sin(wt),t,-x(s)\sin(w t)+z(s)\cos(w t))$ is minimal if and only if the TreamillSled of its profile either is the $x$-axis  and $\phi$ is a helicoid or it is  one of the branches of the hyperbola $\frac{y^2}{M^2}-w^2x^2=1$ for some non zero $M$.
\end{thm}

\begin{proof}
Let us assume that the profile curve $\alpha(s)=(x(s),z(s))$ is parametrized by arc-length. If $TS(\alpha)(s)=(\xi_1(s),\xi_2(s))$ then by Definition \ref{definition of TS} we have

$$\xi_1(s)= -x^\prime(s) x(s)-z^\prime(s) z(s) \com{and} \xi_2(s)= x(s)z^\prime(s)-z(s)x^\prime(s)$$

Since we are assuming that $\alpha$ is parametrized by arc-length, there exists a function $\theta$ such that
$\alpha^\prime(s)=(\cos(\theta(s)),\sin(\theta(s))$. From the previous equation we get that

$$ \theta^\prime(s)= x^\prime(s)z^{\prime\prime}(s)- z^\prime(s)x^{\prime\prime}(s) $$

We this definition of $\theta(s)$ and the definition of the function $\xi_1(s)$ and $\xi_2(s)$ given above, we obtain that

$$x(s)= -\xi_1(x)\cos(\theta(s))+\xi_2(s) \sin(\theta(s)) \com{and} z(s)= - \xi_1(s) \sin(\theta(s))-\xi_2(s)\cos(\theta(s))$$

A direct computation shows that

$$\nu=(\frac{\sin(wt-\theta)}{\sqrt{1+w^2\xi_1^2}},-\frac{w\xi_1}{\sqrt{1+w^2\xi_1^2}},\frac{\cos(wt-\theta)}{\sqrt{1+w^2\xi_1^2}})$$

is a Gauss map of the immersion $\phi$ and, with respect to this Gauss map, the first and second fundamental form are given by

$$ E=1\quad F=-w\xi_2\quad G=1+w^2(\xi_1^2+\xi_2^2)\quad e= \frac{\theta^\prime}{\sqrt{1+w^2\xi_1^2}} \quad f=\frac{-w}{\sqrt{1+w^2\xi_1^2}}\quad g=\frac{w^2\xi_2}{\sqrt{1+w^2\xi_1^2}}$$

Using the values above we get that the mean curvature $H$ of the $\phi$ is given by

$$H=\frac{-w^2\xi_2+\theta^\prime\, (1+w^2(\xi_1^2+\xi_2^2))}{2(1+w^2\xi_1^2)^{\frac{3}{2}}} $$

Therefore, the equation $H=0$, that is, the minimality of the immersion $\phi$,  implies

$$ \theta^\prime=\frac{w^2\xi_2}{1+w^2(\xi_1^2+\xi_2^2)}$$

From the definition of $\xi_1$ and $\xi_2$ we get that

\begin{eqnarray*}
\xi_1^\prime&=&-x^{\prime\prime}x-(x^\prime)^2-z^{\prime\prime}z-(z^\prime)^2\\
            &=&\theta^{\prime}\, x\sin(\theta)-\theta^{\prime}\, z\cos(\theta)-1\\
            &=&\theta^{\prime}\, \xi_2-1
\end{eqnarray*}

Likewise we obtain that $\xi_2^\prime=-\theta^\prime\, \xi_1$. Therefore if $\phi$ is minimal, replacing the expression for $\theta^\prime$ above, we get that

\begin{eqnarray*}
\xi_1^\prime&=& \frac{w^2\xi_2^2}{1+w^2(\xi_1^2+\xi_2^2)}-1\\
\xi_2^\prime&=& -\frac{w^2\xi_1\xi_2}{1+w^2(\xi_1^2+\xi_2^2)}
\end{eqnarray*}

A direct verification shows that if $(\xi_1(s),\xi_2(s))$ satisfies the differential equation above then,

$$\frac{\xi_2(s)}{\sqrt{1+w^2\xi_1(s)^2}}=M\com{for some constant $M$}$$

If $M=0$ then $\xi_2(s)=0$ and $\xi_1(s)=-t+c$. That is, in this case the TreadmillSled of $\alpha$ is the $x$-axis. A direct computation using the inverse formula for the TreadmillSled, see Proposition \ref{inverse formula for TS}, gives us that $\alpha$ is a line through the origin, and therefore $\phi$ is a helicoid. If $M$ is not zero, we get by squaring the centered equation above,  that the TreadmillSled lies in one of the branches of the hyperbola $\frac{y^2}{M^2}-w^2x^2=1$. A direct computation shows that if we parametrize one of these branches with the right orientation, that is, making sure that the second condition of Proposition \ref{inverse formula for TS} holds true, then we will find a parametrization of a profile curve that produces a minimal helicoidal surface. This finish the proof of the theorem.

\end{proof}

\begin{figure}[h]\label{treadmillSled helicoidal minimal}
\centerline{\includegraphics[width=6.54cm,height=5cm]{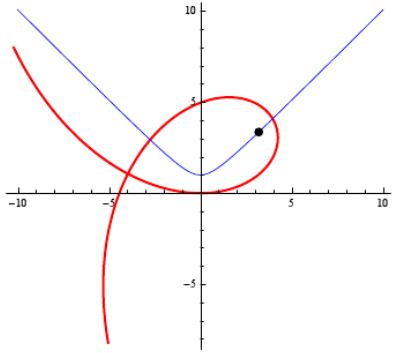}}
\caption{The treadmillSled of the profile curve of a helicoidal minimal surface is either the $x$-axis or a hyperbola. In the notation of Theorem \ref{thm ts of helicoidal minimal}, in this figure, we have $w=1$ and $M=1$}
\end{figure}

\end{document}